\documentclass{amsart} 
\usepackage{amssymb}
\usepackage{amsthm}
\usepackage{dsfont}
\usepackage[curve,matrix,arrow]{xy}
\theoremstyle{plain}\newtheorem{Theorem}{Theorem}[section]
\theoremstyle{plain}
\theoremstyle{plain}\newtheorem{Corollary}[Theorem]{Corollary}
\theoremstyle{plain}\newtheorem{Lemma}[Theorem]{Lemma}
\theoremstyle{plain}\newtheorem{Proposition}[Theorem]{Proposition}
\theoremstyle{definition}\newtheorem{Definition}[Theorem]{Definition}
\theoremstyle{definition}
\theoremstyle{definition}
\theoremstyle{definition}\newtheorem{Remark}[Theorem]{Remark}
\theoremstyle{definition}
\theoremstyle{plain}\newtheorem{Statement}[Theorem]{}

\def\CC{{\mathcal{C}}}    
    
\def\CE{{\mathcal{E}}}

\def\CL{{\mathcal{L}}}  \def\bCL{{\mathbf{L}}}

\def\N{{\mathbb N}}

\def\Z{{\mathbb Z}}

\def\One{{\mathds{1}}}

\def\Aut{\mathrm{Aut}}   \def\bAut{\mathbf{Aut}}

\def\Hom{\mathrm{Hom}}           
  \def\bHom{\mathbf{Hom}}
           
\def\Id{\mathrm{Id}}

\def\Map{\mathrm{Map}}

\def\op{\mathrm{op}}

\def\sym{\mathfrak{S}} 

\title{The operad  of Latin hypercubes } 
\author{Markus Linckelmann} 
\date{\today}

\address{Markus Linckelmann \\
School of Science \& Technology \\
Department of Mathematics \\
City, University of London \\
Northampton Square \\
London EC1V 0HB \\
United Kingdom}
\email{markus.linckelmann.1@city.ac.uk}



\begin{document}

\begin{abstract}
We show that the sets of $d$-dimensional Latin hypercubes over a non-empty
set $X$, with $d$ running over the positive integers, determine an operad which 
is isomorphic to a sub-operad of the endomorphism operad of $X$. We  generalise
this to categories with finite products, and then further to internal versions for certain 
Cartesian closed  monoidal categories with pullbacks. 
\end{abstract}

\maketitle

\section{Introduction} \label{introSection} 

There are several variants of the definition of Latin hypercubes in the literature - see the
discussions and references in \cite[\S 4]{BCCS} and \cite[Section 1]{McKayWan}.
In order to define the version we will be considering here, we fix the
following notation and conventions.
Categories are assumed to be essentially small.
Let $X$ be an object in a category $\CC$ with finite products, and let $d$ be a positive 
integer. The product $X^d$ of $d$ copies of $X$ in $\CC$
is equipped with $d$ canonical projections $\pi^d_i : X^d\to X$,
characterised by  the universal 
property that for any object $Y$ and morphisms $\eta_i : Y\to X$,
with $1\leq i\leq d$, there is a unique morphism $\eta : Y\to X^d$ satisfying
$\eta_i = \pi^d_i\circ \eta$ for $1 \leq i\leq d$. In that case we write 
$\eta=$ $(\eta_i)_{1\leq i\leq d}$.
Applied to $X^{d+1}$ instead of $Y$
and any subset  of $d$ of the $d+1$ canonical projections  $\pi^{d+1}_j$ yields
for every $s$ such that $1\leq s \leq d+1$ a unique morphism 
$\tau^{d+1}_s : X^{d+1} \to X^d$
satisfying 
$\pi^{d+1}_i = \pi^d_i\circ \tau^{d+1}_s$
for $1\leq i\leq s-1$, and
$\pi^{d+1}_i = \pi^d_{i-1}\circ\tau^{d+1}_s$
for $s+1 \leq i\leq d+1$. With the notation above, this is equivalent to
$$\tau^{d+1}_s= (\pi^{d+1}_1,..,\pi^{d+1}_{s-1},
\pi^{d+1}_{s+1},..,\pi^{d+1}_{d+1}). $$
If $X$ is a non-empty set, then $X^{d+1}$ is the Cartesian product of $d+1$ copies of $X$, 
and we have 
$$\tau^{d+1}_s(x_1,x_2,..,x_{d+1})= (x_1,..,x_{s-1}, x_{s+1},..,x_{d+1}),$$ 
where $x_i\in X$ for $1\leq i\leq d+1$. That is,  $\tau^{d+1}_s$
is the projection from $X^{d+1}$ to $X^d$ which discards the coordinate $s$ in $X^{d+1}$.

\begin{Definition} \label{hypercubeDef}
Let $\CC$ be a category with finite products,  let $X$ be an object in $\CC$, and let $d$ be 
a positive integer.
A {\em Latin hypercube of dimension $d$ over $X$} is a morphism $\lambda : L \to X^{d+1}$
in $\CC$ such that $\tau^{d+1}_s\circ \lambda : L \to X^d$ is an isomorphism
in $\CC$, for $1\leq s\leq d+1$.
\end{Definition} 

The morphism $\lambda$ in this Definition is necessarily a monomorphism. 
Two Latin hypercubes $\lambda : L\to X^{d+1}$ and $\lambda' : L'\to X^{d+1}$ 
are called {\em isomorphic} if there is an isomorphism $\alpha : L\to L'$ such that 
$\lambda=\lambda'\circ\alpha$. In that case $\alpha$ is unique since $\lambda'$ 
is a monomorphism. 

\medskip
A Latin hypercube of  dimension $d$ over a non-empty set $X$ is uniquely 
isomorphic to a Latin hypercube given by
 a  subset $L$ of $X^{d+1}$ with the property that for any choice 
of $d$ of the $d+1$ coordinates of an element in $X^{d+1}$ there is a unique element 
in the remaining coordinate
such that the resulting element belongs to $L$. In particular, we can 
identify a Latin hypercube of dimension $d$ over a non-empty set uniquely, up to
unique isomorphism, as the graph  of the 
map $f : X^d\to X$ satisfying $(x_1,x_2,..,x_d, f(x_1,x_2,..,x_d))\in L$ for all
$(x_1,x_2,..,x_d)\in X^d$. In this way 
Latin hypercubes of dimension $d$ can be identified with a subset, denoted 
$\CL(X^d,X)$, of  the set $\Map(X^d,X)$ of all maps from $X^d$ to $X$. 
Any map in $\CL(X^d,X)$ is  clearly surjective. The following result
 shows that the subsets $\CL(X^d,X)$ of $\Map(X^d,X)$ with $d\in \N$
determine an operad. We refer to
\cite{May72} or \cite[Part II, \S 1.2]{MarShnSta}  for basic terminology on operads.

\begin{Theorem} \label{thm1}
Let $X$ be a non-empty set. There is a sub-operad $\CL$ of the endomorphism
operad $\CE$ of $X$ such that $\CL(d)=$ $\CL(X^d,X)$, for all  $d\in \N$.
\end{Theorem}

This is not the most general version in which this result can be stated.
Theorem \ref{thm1}  admits a generalisation to  categories with finite products, 
which we will describe in Theorem \ref{thm2}, and 
 a further generalisation, in Theorem \ref{thm4},  to certain Cartesian closed 
monoidal categories  in which  $\CL$ can  be identified with a 
suboperad of an internal endomorphism operad.
We have chosen to state and prove this result first in the context of non-empty 
sets in order to not distract  from the elementary nature  of the proof.
Operads were first introduced for topological spaces, and Theorem \ref{thm1} 
holds verbatim for non-empty  compactly generated topological spaces (this is
a special case of Theorem \ref{thm4}; see Remark \ref{Rem-thm4}).

\begin{Remark}
As  pointed out in \cite{McKayWan},  if $d=1$, then a Latin hypercube over  a
non-empty set $X$  is a subset of $X^2$ of the form $\{(x,\sigma(x))\}_{x\in X}$ 
for some permutation  $\sigma$ of $X$, and hence Latin hypercubes of dimension 
$1$ over $X$ correspond  to the elements of the symmetric group $\sym_X$ of 
permutations of $X$, so they form themselves a group. This group structure is 
encoded as the structural map  $-\circ_1 -$ on $\CL(1)$ of the operad $\CL$. 
\end{Remark}

\begin{Remark} \label{AutLatinRemark}
Let $X$ be an object in a category $\CC$ with finite products and let $d$ be a 
positive integer.  The group $\Aut_\CC(X)\wr \sym_{d+1}$ acts on $X^{d+1}$, hence 
on the class of Latin hypercubes of  dimension $d$ over $X$, with the base group of 
$d+1$ copies of $\Aut_\CC(X)$ acting by composing $\lambda$ with a $(d+1)$-tuple 
of automorphisms of $X^{d+1}$ and $\sym_{d+1}$ acting on $X^{d+1}$ by permuting 
the $d+1$ canonical  projections  $\pi^{d+1}_i : X^{d+1}\to X$. This action induces an 
action of $\Aut_\CC(X)\wr \sym_{d+1}$ on the isomorphism classes of Latin 
hypercubes of dimension $d$ over $X$, which for $\CC$ the category of sets is
the standard notion of paratopism.
\end{Remark}

\begin{Remark}
Let $X$ be an object in a category $\CC$ with finite products, and let $d$ be a
positive integer. The canonical projections $\pi^{d}_i : X^{d}\to$ $X$ are
split surjective, with section the diagonal morphism $\delta : X\to$ $X^{d}$
defined as the unique morphism such that $\pi^{d}_i\circ\delta=$ $\Id_X$, for
$1\leq i\leq d$.  Let $\lambda : L\to X^{d+1}$ be a Latin hypercube.
The $d+1$ components $\lambda_i=$ $\pi^{d+1}_i\circ\lambda : L\to X$ 
of $\lambda$  are split  epimorphisms. 
Indeed, if we choose $s\neq i$, with $1\leq i,s\leq d+1$,
then $\pi^{d+1}_i$ factors through $\tau^{d+1}_s$; more precisely, 
$\pi^{d+1}_i\circ\lambda = \pi^{d}_j\circ \tau^{d+1}_s\circ \lambda$
where $j=i$ if $i<s$, and $j=i-1$ if $i>s$. Since $\tau^{d+1}_s\circ\lambda$
is an isomorphism and $\pi^d_j$ a split epimorphism, it follows that their
composition is a split epimorphism, and hence so is $\pi^{d+1}_i\circ\lambda$.
\end{Remark}

\section{Proof of Theorem \ref{thm1}} \label{thm1proofSection}

The endomorphism operad $\CE$ of a non-empty set $X$ consists
of the sets $\CE(n) = \Map(X^n,X)$  for any positive integer $n$, 
together with strutural maps
$$-\circ_i - : \Map(X^n,X) \times \Map(X^m,X) \to \Map(X^{n+m-1},X)$$
given by
$$(f \circ_i g)(x_1,x_2,..., x_{n+m-1}) =
f(x_1,..,x_{i-1}, g(x_i,..,x_{i+m-1}), x_{i+m},..,x_{n+m-1})$$
for all positive integers $n$, $m$, $i$, such that $1\leq i\leq n$,  all
$x_1,x_2,..,x_{n+m-1}\in X$, and all maps $f\in \Map(X^n,X)$ and $g\in \Map(X^m,X)$. 
The sets $\CE(n)=\Map(X^n,X)$ are equipped with
the action of $\sym_n$ on the $n$ coordinates of $X^n$, and the identity map
$\Id_X\in$ $\CE(1)=\Map(X,X)$ is the unit element of this operad.
For the associativity properties of the maps $-\circ_i -$ and their compatibility with 
the symmetric group actions on the sets $\Map(X^n,X)$, see for instance 
\cite[Definition 1.2]{May72} or \cite[Part II,\S 1.2]{MarShnSta}.
The main step for the proof of Theorem \ref{thm1} is the following Lemma.

\begin{Lemma}  \label{Lem-closed}
Let $X$ be a non-empty set, and 
let $d$, $e$, $i$ be positive integers such that $1\leq i\leq d$. 
Let $f\in\CL(X^d,X)$ and $g\in\CL(X^e,X)$. 
Then $f\circ_i g\in$ $\CL(X^{d+e-1},X)$.
\end{Lemma}

\begin{proof}
In order to show that
$f\circ_i g$ belongs to $\CL(X^{d+e-1},X)$ we need to show that for any $c\in X$
and an arbitrary choice of $d+e-2$ of the $d+e-1$ elements $x_1,x_2,..,x_{d+e-1}\in X$
the remaining of these elements is uniquely determined by the equation
\begin{equation}  \label{eq-1}
f(x_1,..,x_{i-1}, g(x_i,..,x_{i+e-1}), x_{i+e},..,x_{d+e-1}) = c.
\end{equation}
Let $s$ be an integer such that $1\leq s\leq d+e-1$. Fix elements
$x_1,x_2,..,x_{s-1}, x_{s+1},..,x_{d+e-1}, c \in$ $X$. 

Consider first that case where $s\leq i-1$ or $s\geq i+e$.
Then, setting $y=g(x_i,..,x_{i+e-1})$, the
Equation \ref{eq-1} becomes
\begin{equation} \label{eq-2}
f(x_1,..,x_{i-1},y,x_{i+e},..,x_{d+e-1}) = c.
\end{equation}
All entries  but the entry $x_s$ in this equation are fixed. Since $f\in\CL(X^d,X)$ it follows
that there is a unique choice for $x_s\in X$ such that Equation \ref{eq-2} holds, and
hence a unique choice for $x_s\in X$ such that Equation \ref{eq-1} holds.

Consider the remaining case where $i\leq s \leq i+e-1$. Then in particular the
elements $x_1,..,x_{i-1}, x_{i+e},..,x_{d+e-1}$ are fixed in $X$.
Since $f\in\CL(X^d,X)$ it follows that there is a unique $y\in X$ such that
Equation \ref{eq-2} holds.  Thus for Equation \ref{eq-1} to hold we must have
\begin{equation} \label{eq-3}
g(x_i,..,x_{i+e-1}) = y.
\end{equation}
In this equation all but $x_s$ have been chosen. Since $g\in\CL(X^e,X)$ it follows
that there is a unique choice $x_s\in X$ such that Equation \ref{eq-3} holds.
In all cases there is a unique choice of $x_s$ such that Equation \ref{eq-1} holds.
This shows that $f\circ_i g$ belongs to $\CL(X^{d+e-1},X)$ and completes the proof.
\end{proof}

\begin{proof}[{Proof of Theorem \ref{thm1}}]
Let $d\in \N$.  A map $f\in\CE(d)=$ $\Map(X^d,X)$ 
belongs to $\CL(d)=\CL(X^d,X)$ if and only if  the set 
$$L = \{(x_1,x_2,..,x_d, f(x_1,x_2,..,x_d))\ |\ x_1,x_2,..,x_d \in X\}$$
is a Latin hypercube in $X^{d+1}$. Equivalently, $f$ belongs to $\CL(d)$ if and only
for every  $c\in X$ and an arbitrary  choice of $d-1$ of the $d$ 
entries $x_1,..,x_d\in X$ the remaining entry is uniquely determined by the equation
$f(x_1,x_2,..,x_d)=c$. 
The action of $\sym_d$ on $\Map(X^d,X)$
by permuting the $d$ coordinates  of $X^d$ clearly preserves the subset $\CL(d)$, 
and $\Id_X$ belongs to $\CL(1)$.
In order to prove Theorem \ref{thm1} it
remains to show that the sets $\CL(d)$ are closed under the operations $-\circ_i-$.
This is done in Lemma \ref{Lem-closed} above, and
this concludes the proof.
\end{proof}

\begin{Remark} \label{rem1} 
We could have very slightly simplified the proof of Lemma \ref{Lem-closed}
by observing that thanks to the symmetric group actions on the coordinates 
of  Latin hypercubes  (cf. Remark \ref{AutLatinRemark}) it would have been 
sufficient in the proof of Lemma \ref{Lem-closed}  to consider the map 
$f\circ_{d} g$ and a single $s$ in the  distinction into the  two cases  for $s$.  
That is,  it would have been sufficient in the last part of the proof of 
Lemma \ref{Lem-closed}  to consider the cases where  either $1=s\leq d-1$ or 
$s=d$.  We will make use of this observation in the proof of the more general 
Theorem \ref{thm2} below.
\end{Remark}

\begin{Remark} \label{rem2}
The proof of Theorem \ref{thm1}, as written, involves choices of  elements in the 
set $X$.  In Section \ref{monoidalSection} below we will rewrite this proof in such 
a way that it extends to categories with finite products. 
 \end{Remark}

\section{On Latin hypercubes in  Cartesian  monoidal categories} \label{monoidalSection} 

In this section we extend Theorem \ref{thm1} to categories with finite products.
Given an object $X$ in a category $\CC$ with finite products, 
a positive integer $n$, and a morphism $\lambda : L \to X^n$ in $\CC$, we denote
as at the beginning  by $\lambda_i=$ $\pi^n_i\circ\lambda$ the composition of 
$\lambda$ with the $i$-th canonical projection $\pi^n_i : X^n\to X$, 
where $1\leq i\leq n$. The morphism $\lambda$ is
uniquely determined by the $\lambda_i$, and we will write abusively $\lambda=$
$(\lambda_i)_{1\leq i\leq n}$ whenever convenient. 
If $\sigma\in$ $\sym_n$, then $\sigma$ induces an automorphism $\hat\sigma$ 
on $X^n$ given by $\hat\sigma_i=\pi^n_{\sigma^{-1}(i)}$; that is, $\hat\sigma$ 
permutes the coordinates of $X^n$. This yields a group
homomorphism $\sym_n \to \Aut_\CC(X^n)$.  We write ${^\sigma{\lambda}}=$ 
$\hat\sigma\circ\lambda$. 
 We extend the earlier notation 
$\CL_\CC(X^d,X)$ in the obvious way.

\begin{Definition} \label{CL-DEF}
Let $\CC$ be a category with finite products.
Let $d$ be a positive integer.
We denote by $\CL_\CC(X^d,X)$ the subset of $\Hom_\CC(X^d,X)$
consisting of all morphisms $f : X^d\to X$ such that
the morphism $(\Id_{X^d}, f) : X^d \to X^d\times X = X^{d+1}$ is a Latin hypercube.
\end{Definition}

We first identify canonical representatives in isomorphism classes of Latin hypercubes.

\begin{Proposition} \label{monoidalProp1}
Let $\CC$ be a category with finite products, let $X$ be an object in $\CC$, and
let $d$, $s$ be  positive integers such that $1\leq s\leq d+1$. 
Let $\lambda: L\to X^{d+1}$ be a Latin hypercube.
 Then there is a unique Latin hypercube $\iota : X^d \to X^{d+1}$ such that
$\tau^{d+1}_{s}\circ \iota = \Id_{X^d}$ and such that 
$\iota\circ\alpha=\lambda$ for some isomorphism $\alpha : L \to X^d$.
In that case we have  $\alpha = $ $\tau^{d+1}_{s}\circ \lambda :  L \to X^d$. 
\end{Proposition}

\begin{proof}
By the definition of Latin hypercubes, the morphism 
$\alpha = \tau^{d+1}_{s}\circ \lambda : L\to$ $X^d$ is an isomorphism. 
Then setting $\iota = \lambda\circ \alpha^{-1}$ implies immediately that $\alpha$
determines an isomorphism between the Latin hypercubes $\lambda : L\to X^{d+1}$
and $\iota : X^d\to X^{d+1}$. We need to show that $\alpha$ and $\iota$  are
unique subject to these properties. Let $\iota' : X^d\to X^{d+1}$ a Latin hypercube
and $\alpha' : L\to X^d$ an isomorphism such that $\tau^{d+1}_{s}\circ\iota'=$
$\Id_{X^d}$ and such that $\iota'\circ\alpha'= \lambda$. Composing this
equality with $\tau^{d+1}_{s}$ yields 
$$\alpha' = \Id_{X^d}\circ\alpha' = \tau^{d+1}_{s}\circ\iota'\circ \alpha'
= \tau^{d+1}_{s}\circ \lambda = \alpha$$
This implies $\iota'=\lambda\circ\alpha^{-1}=\iota$, whence the uniqueness of $\iota$ 
and $\alpha$ as stated. The result follows.
\end{proof}

Applied with $s=d+1$, Proposition \ref{monoidalProp1} implies that  -  as 
earlier in the category of sets - the  morphisms in $\CL_\CC(X^d,X)$
parametrise the isomorphism classes of $d$-dimensional Latin hypercubes over $X$.

\begin{Corollary} \label{monoidalCor1}
Let $\CC$ be a category with finite products, let $X$ be an object in $\CC$, and
let $d$ be a positive integer. Any Latin hypercube of dimension $d$ over $X$ is
uniquely isomorphic to  a Latin hypercube of the form $\iota : X^d\to X^{d+1}$ 
such that $\iota$ is a section of the morphism $\tau^{d+1}_{d+1}$ discarding the 
coordinate $d+1$. Any such morphism  $\iota$ is then uniquely determined by
its last component $f = \iota_{d+1} = \pi^{d+1}_{d+1}\circ\iota  : X^d\to X$.
In particular, the set $\CL_\CC(X^d,X)$ 
parametrises the  isomorphism classes of $d$-dimensional Latin hypercubes over $X$.
\end{Corollary}

We show now that the sets $\CL_\CC(X^d,X)$
form an operad, together with the structural maps
$-\circ_i -$ defined as follows.
Let $f\in$ $\CL_\CC(X^d,X)$ and $g\in$ $\CL_\CC(X^e,X)$, where
$d$, $e$ are positive integers. For $1\leq i \leq d$, the structural map
$$-\circ_i - : \Hom_\CC(X^d,X) \times \Hom_\CC(X^e,X) \to \Hom_\CC(X^{d+e-1},X)$$
sends $(f, g)$ to the morphism 
$$f \circ (\Id_{X^{i-1}}\times g \times \Id_{X^{d-i}})$$
where we identify $X^{i-1}\times X^e\times X^{d-i}=X^{d+e-1}$ for the domain
of  this morphism and where we identify $X^{i-1}\times X\times X^{d-i}=$ $X^d$
for the codomain of $ \Id_{X^{i-1}}\times g \times \Id_{X^{d-i}}$.
One checks that if $X$ is a set, this coincides with the earlier definition of $f\circ_i g$.

\begin{Theorem} \label{thm2}
Let $X$ be an object in a category $\CC$ with finite  products. 
There is a sub-operad $\CL$ of the endomorphism set
operad $\CE$ of $X$ such that $\CL(d)=$ $\CL_\CC(X^d,X)$, for all  $d\in \N$.
In particular, for any positive integers $d$, $e$, $i$ such that $1\leq i\leq d$,
and any morphisms $f\in$ $\CL_\CC(X^d,X)$ and $g\in$ $\CL_\CC(X^e,X)$ we have
$f\circ_i g \in \CL_\CC(X^{d+e-1},X)$. 
\end{Theorem}

\begin{proof}
The proof amounts to rewriting the proof of Theorem \ref{thm1}, including the statement and proof of
Lemma \ref{Lem-closed}, in such a way
that all steps remain valid for the Cartesian products in $\CC$.  In order to keep
this readable, we mention in each step what this corresponds to in the case where
$X$ is a non-empty set (and by considering coordinates, one easily translates this
to statements in $\CC$). 

\medskip
Let $f\in$ $\CL_\CC(X^d,X)$ and $g\in$ $\CL_\CC(X^e,X)$, where
$d$, $e$ are positive integers. That is, the morphisms
$$(\Id_{X^d}, f) : X^d \to X^{d+1}, $$
$$(\Id_{X^e},g) : X^e \to X^{e+1}$$
are Latin hypercubes. 
As in the proof of Theorem \ref{thm1}, the unitality and symmetric group actions
are obvious, and we only need to show, analogously to Lemma \ref{Lem-closed}, 
that  $(\Id_{X^{d+e-1}}, f\circ_i g)$ is a Latin hypercube. That is, we need to show that for
$1\leq s\leq d+e-1$ and $1\leq i\leq d$, 
the composition $\tau^{d+e}_s \circ  (\Id_{X^{d+e-1}}, f\circ_i g)$
is an automorphism of $X^{d+e-1}$. As pointed out in Remark \ref{rem1}, since
we may permute coordinates, it suffices to do this for $i=d$ and  in the two cases where
either $1=s\leq d-1$ or $s=d$.

\medskip
We consider first the case $1=s\leq d-1$, so $d\geq 2$. We need to show that the morphism
$$\tau^{d+e}_1\circ (\Id_{X^{d+e-1}}, f\circ_d g)$$
is an automorphism of $X^{d+e-1}$. If $X$ is a set, then  this automorphism is given by 
the assignment
$$(x_1,x_2,..,x_{d+e-1}) \mapsto (x_2,..,x_{d+e-1}, f(x_1,..,x_{d-1},g(x_d,..,x_{d+e-1}))).$$
First, the morphism $\tau^d_1\circ (\Id_{X^d}, f)$ is an automorphism of $X^d$ because
$(\Id_{X^d},f)$ is a Latin hypercube. We note that if $X$ is a non-empty set, then the morphism 
$\tau^d_1\circ (\Id_{X^d}, f)$ is given by the assignment 
$$(x_1,..,x_d) \mapsto (x_2,..,x_d, f(x_1,..,x_d)).$$
Compose this with the automorphism $\hat\sigma$
induced by the cyclic permutation $\sigma=$ $(1,2,..,d)$ on coordinates. 
The resulting automorphism
$$\hat\sigma\circ\tau^d_1\circ (\Id_{X^d},f)$$
is, for $X$ a set, given by the assignment
$$(x_1,..,x_d) \mapsto (f(x_1,..,x_d),x_2,..,x_d)$$
The Cartesian product  of this automorphism with $-\times \Id_{X^{e-1}}$ yields
an automorphism of $X^{d+e-1}$, which for $X$ a set corresponds to 
$$(x_1,..,x_{d+e-1}) \mapsto (f(x_1,..,x_d),x_2,..,x_{d+e-1}))$$
Again permuting cyclically the coordinates yields an automorphism of $X^{d+e-1}$ 
which we will denote by $\alpha$, which, if
$X$ is a set,  corresponds to
$$(x_1,..,x_{d+e-1}) \mapsto (x_2,..,x_{d+e-1}, f(x_1,..,x_d)).$$
Using the fact that $\tau^e_1\circ (\Id_{X^e}, g)$ is an automorphism of $X^e$,
combined with cyclically permuting coordinates, we obtain an automorphism
$\gamma$ of $X^e$ which corresponds to the assignement
$$(x_d,..,x_{d+e-1}) \mapsto (g(x_d,..,x_{d+e-1}), x_{d+1},..,x_{d+e-1}).$$
Then $\beta = \Id_{X^{d-1}}\times \eta$ is the automorphism of $X^{d+e-1}$
which corresponds to
$$(x_1,..,x_{d+e-1}) \mapsto (x_1,..,x_{d-1}, g(x_d,..,x_{d+e-1}), x_{d+1},..,x_{d+e-1}).$$
Similarly, $\gamma=\Id_{X^{d-1}}\times \eta \times \Id_X$ is an automorphism of 
$X^{d+e-1}$.
Now $\alpha\circ\beta$ is an automorphism of $X^{d+e-1}$ corresponding to
$$(x_1,..,x_{d+e-1}) \mapsto (x_2,..,x_{d-1},g(x_d,..,x_{d+e-1}),x_{d+1},..,x_{d+e-1},
f(x_1,..,x_{d-1}, g(x_d,..,x_{d+e-1}))).$$
Thus the automorphism $\gamma^{-1}\circ\alpha\circ\beta$ of $X^{d+e-1}$
coincides with the morphism $\tau^{d+e}_1\circ (\Id_{X^{d+e-1}}, f\circ_d g)$.
This proves the result in  the case $1=s\leq d-1$.

\medskip
Consider next the case $s=d$. We need to show that the morphism
$\tau^{d+e}_d\circ (\Id_{X^{d+e-1}}, f\circ_d g)$ is an automorphism of $X^{d+e-1}$.
If $X$ is a set, then this automorphism is given by the assignment
$$(x_1,..,x_{d+e-1}) \mapsto   
(x_1,..,x_{d-1},x_{d+1},..,x_{d+e-1}, f(x_1,..,x_{d-1},g(x_d,..,x_{d+e-1}))).$$
As before, by using the automorphism $\tau^e_1\circ (\Id_{X^e}, g)$,
cyclically permuting the  coordinates and then applying $\Id_{X^{d-1}}\times-$
 we obtain anautomorphism $\delta$ of $X^{d+e-1}$, which for $X$ a set
 corresponds to
 $$(x_1,..,x_{d+e-1}) \mapsto (x_1,..,x_{d-1}, g(x_d,..,x_{d+e-1}), x_{d+1},..,x_{d+e-1}).$$
 Similarly, applying $-\times \Id_{X^{e-1}}$ to the automorphism 
 $\tau^d_d\circ (\Id_{X^d}, f)$ yields an automorphism $\epsilon$ of $X^{d+e-1}$,
 which for $X$ a set corresponds to
 $$(x_1,..,x_{d+e-1}) \mapsto (x_1,..,x_{d-1}, f(x_1,..,x_d), x_{d+1},..,x_{d+e-1}).$$
 Thus $\epsilon\circ\delta$ is the automorphism of $X^{d+e-1}$ which for $X$ a set
 corresponds to 
 $$(x_1,...,x_{d+e-1}) \mapsto (x_1,..,x_{d-1},x_{d+1},..,x_{d+e-1}, (f\circ_d g)(x_1,..,x_{d+e-1})),$$
 and this is indeed the automorphism $\tau^{d+e-1}_d\circ (\Id_{X^{d+e-1}}, f\circ_d g)$.
 This proves the second case, and the result follows.
\end{proof}

\section{On Latin hypercubes in closed Cartesian monoidal categories}
\label{closedSection}

A monoidal category  $\CC$ with unit object $\One$  is closed if $\CC$ has an internal Hom, denoted 
$\bHom$. That is, $\bHom : \CC^\op\times \CC\to \CC$
is a bifunctor such that for any object $X$ in $\CC$
the functor $X\times -$ on $\CC$ is left adjoint to the functor $\bHom(X,-)$. 
This adjunction yields in particular natural bijections
$\Hom_\CC(\One, \bHom(X,Y))\cong$ $\Hom_\CC(X,Y)$ and natural
isomorphisms $\bHom(\One, X)\cong$ $X$;
see Kelly \cite{Kelly82} and \cite{MacMoe} 
for more background material. Following \cite{Kelly05}, 
endomorphism operads  can be defined over objects in certain closed symmetric 
monoidal categories.
For the definition of Latin hypercubes we need in addition that  
$\CC$ is Cartesian monoidal; that is, the monoidal product is a product in the
category $\CC$. In that case the unit object $\One$ is a terminal object in $\CC$.   
In what follows we say that a morphism $\hat\alpha$ between Hom objects in $\CC$ 
lifts a map $\alpha$ if $\alpha$ is the image of 
$\hat\alpha$ under the functor $\Hom_\CC(\One,-)$ modulo canonical identifications.

In order to show that the morphism sets $\CL_\CC(X^d,X)$ lift to internal objects whenever
$\CC$ has pullbacks,  we will need, from 
 \cite[Exercise 5, page 213]{MacMoe}, the fact that automorphism groups of objects lift to
 internal objects.  We have a pullback diagram
$$\xymatrix{\Aut_\CC(X) \ar[rr] \ar[d]_{\delta} & & \{\Id_\One\} \ar[d] \\
\Hom_\CC(X,X)\times \Hom_\CC(X,X) \ar[rr]_{\mu} & & \Hom_\CC(X,X)\times \Hom_\CC(X,X) 
}$$
where $\delta$ sends $\sigma\in\Aut_\CC(X)$ to $(\sigma,\sigma^{-1})$, where
$\mu(\alpha,\beta)=$ $(\beta\circ\alpha, \alpha\circ\beta)$ for any 
$\alpha$, $\beta\in$ $\Hom_\CC(X,X)$, and where the right vertical map
sends $\Id_\One$ to $(\Id_X,\Id_X)$. The lower horizontal map $\mu$ commutes with the
involution on $\Hom_\CC(X,X)\times \Hom_\CC(X,X)$ given by exchanging coordinates.
The map $\mu$ and the right
vertical map lift to maps on internal Hom objects, and hence, if $\CC$ has pullbacks, then
the above diagram lifts to a pullback diagram in $\CC$ of the form
\begin{Statement} \label{bAut} 
$$\xymatrix{
\bAut(X) \ar[rr] \ar[d]_{(\gamma,\gamma')} & & \One = \One\times \One \ar[d]^{\iota} \\
\bHom(X,X)\times \bHom(X,X) \ar[rr]_{\nu} & & \bHom(X,X)\times \bHom(X,X)
}$$
\end{Statement} 
Composing $\iota$ with the canonical involution on $\bHom(X,X)\times\bHom(X,X)$ commutes
with $\nu$, 
does not change $\iota$, while it changes $(\gamma, \gamma')$ to $(\gamma',\gamma)$.
The universal property of pullbacks implies that there is a unique automorphism
$\epsilon$ of the object $\bAut(X)$ of order $2$ with the property that
$(\gamma',\gamma)=$ $(\gamma,\gamma')\circ\epsilon$. This automorphism lifts
the bijection given  by taking inverses in the group $\Aut_\CC(X)$. Since $\iota$ is
trivially a monomorphism, it follows that $(\gamma,\gamma')$ is a monomorphism.
If $\Hom_\CC(\One,-)$ is faithful, hence reflects monomorphism, then both
 $\gamma$ and $\gamma'$ are  monomorphisms, since they lift inclusion maps.
The following result shows that there are internal objects lifting the morphism sets $\CL_\CC(X^d,X)$.

\begin{Theorem} \label{thm3}
Let $\CC$ be a Cartesian closed monoidal category with pullbacks. Let $X$ be an object in $\CC$,
and let $d$ be a  positive integer. 
There is an object $\bCL(X^d,X)$ in $\CC$, determined uniquely up to unique isomorphism,
such that we have a canonical
isomorphism $\Hom_\CC(\One,\bCL(X^d,X))\cong$ $\CL(X^d,X)$, and such that there
is a canonical morphism $\bCL(X^d,X) \to \bHom(X^d,X)$ in $\CC$ which 
lifts the inclusion $\CL_\CC(X^d,X) \to \Hom_\CC(X^d,X)$. If in addition the functor 
$\Hom_\CC(\One,-)$ faithful, then the canonical morphism $\bCL(X^d,X) \to \bHom(X^d,X)$ 
is a monomorphism.
\end{Theorem}

We will need the following characterisation of the morphism sets $\CL_\CC(X^d,X)$
in a category with finite products.

\begin{Lemma} \label{lem-3}
Let $\CC$ be a category with finite products, let $X$ be an object in $\CC$, and let
$d$ be a positive integer. We have a pullback diagram of sets
$$
\xymatrix{\CL_\CC(X^d,X) \ar[d] \ar[rr] & & \prod_{i=1}^d \Aut_\CC(X^d) \ar[d]^\iota \\
\Hom_\CC(X^d,X) \ar[rr]_\gamma & & \prod_{i=1}^d \Hom_\CC(X^d,X^d) }
$$
where $\gamma=$ 
$(\gamma_i : \Hom_\CC(X^d,X)\to \Hom_\CC(X^d,X^d))_{1\leq i\leq d}$ is defined by 
$$\gamma_i(\lambda)=\tau^{d+1}_i \circ (\Id_{X^d},\lambda)$$
for $1\leq i\leq d$ and $\lambda \in \Hom_\CC(X^d, X)$, and where $\iota$ is
the product of $d$ copies of the inclusion $\Aut_\CC(X^d)\to$ $\Hom_\CC(X^d,X^d)$.
\end{Lemma}

\begin{proof}
Let $\lambda\in$ $\Hom_\CC(X^d,X)$.  By definition, we have
$\lambda\in\CL_\CC(X^d, X)$ if and only if $\gamma_i\in$ $\Aut_\CC(X^d)$ for
$1\leq i\leq d$.  This is clearly equivalent  to the assertion that the diagram in the
statement is a pullback diagram.
\end{proof}

\begin{proof}[Proof of Theorem \ref{thm3}]
As described in the diagram \ref{bAut}, appplied with $X^d$ instead of $X$, there is a morphism 
 $\bAut(X^d) \to \bHom(X^d,X^d)$ which lifts the inclusion $\Aut_\CC(X^d)\to$
$\Hom_\CC(X^d,X^d)$.
Both maps $\gamma$ and $\iota$ in the diagram from Lemma \ref{lem-3}  lift to
morphisms $\hat\gamma$ and $\hat\iota$ between the relevant
 internal Hom objects, and hence, by the assumptions on $\CC$,
 there is a pullback diagram  in $\CC$ of the form
 \begin{Statement} \label{LinternalDef}
 $$\xymatrix{\bCL(X^d,X) \ar[d] \ar[rr] & & \prod_{i=1}^d \bAut(X^d) \ar[d]^{\hat\iota} \\
\bHom(X^d,X) \ar[rr]_{\hat\gamma}  & & \prod_{i=1}^d \bHom(X^d,X^d)}$$
\end{Statement}
The functor $\Hom_\CC(\One,-)$ from $\CC$ to the category of sets preserves
pullbacks, hence sends this pullback diagram to a diagram isomorphic to that in
Lemma \ref{lem-3}. It follows in particular that $\Hom_\CC(\One, \bCL(X^d,X))\cong$ 
$\CL_\CC(X^d,X)$. The uniqueness statement follows from the fact that
 pullbacks are unique up to unique isomorphism. If the functor $\Hom_\CC(\One,-)$ is faithful, then
 this functor reflects monomorphisms, whence the last statement.
follows.
\end{proof}

\begin{Theorem} \label{thm4} 
Let $\CC$ be a Cartesian closed monoidal category with pullbacks. 
Suppose that the functor $\Hom_\CC(\One,-)$ is faithful. 
Suppose in addition  that 
for any two objects $Y, $Z$ $ in $\CC$ and any morphism $\zeta :  Z \to \bHom(Y,Y)$, 
if the map $\Hom_\CC(\One,\zeta)$ factors through the inclusion $\Aut_\CC(Y)\to$
$\Hom_\CC(Y,Y)$, then the morphism
$\zeta$ factors through the morphism $\bAut(Y) \to $ $ \bHom(Y,Y)$.
Let $X$ be an object in $\CC$. For any positive integer $d$  the morphism
$\bCL(X^d,X) \to \bHom(X^d,X)$ is a monomorphism, and, with $d$ running over $\N$, these
monomorphisms form a suboperad of the internal endomorphism operad of $X$  in $\CC$.
\end{Theorem}

\begin{proof} 
Note that since we assume $\Hom_\CC(\One,-)$ to be faithful (hence
reflecting monomorphisms), it follows that the morphism $\bCL(X^d,X)\to$ $\bHom(X^d,X)$
from Theorem \ref{thm3}  is a monomorphism, for any positive integer $d$.
Furthermore, as in the proofs of Theorems \ref{thm1}, \ref{thm2}, 
showing the unitality and  compatibility with symmetric group actions is
straightforward. What remains to be proved is that the maps $- \circ_i -$ of the 
endomorphism operad induce maps on the subobjects $\bCL(X^d,X)$ of the
internal Hom objects $\bHom(X^d,X)$. Let $d$, $e$, $i$  be positive integers such that
$1\leq i\leq d$. The map
$$ - \circ_i - : \Hom_\CC(X^d,X) \times \Hom_\CC(X^e,X) \to \Hom_\CC(X^{d+e-1},X)$$
sends $(f,g)$ to $f\circ (\Id_{X^{i-1}} \times g \times \Id_{X^{d-i}})$. Since this involves
composition and products only, this map lifts to a map of internal Hom objects
$$ \bHom(X^d,X) \times \bHom(X^e,X) \to \bHom(X^{d+e-1}, X).$$  
At the level of morphism
sets it follows from Theorem \ref{thm2} that we have a commutative diagram
\begin{Statement} \label{lift} 
$$ \xymatrix{
\Hom_\CC(X^d,X) \times \Hom_\CC(X^e,X) \ar[rr]^{-\circ_i - } & & \Hom_\CC(X^{d+e-1},X) \\
\CL_\CC(X^d,X) \times \CL_\CC(X^e,X) \ar[u] \ar[rr] & & \CL_\CC(X^{d+e-1},X) \ar[u] 
},$$
\end{Statement} 
where the vertical maps are inclusions.
We need to show that this diagram lifts to the internal Hom objects and relevant subobjects. We note that
the vertical maps in the diagram \ref{lift} lift by Theorem \ref{thm3}, and the top horizontal map lifts
by the discussion preceding the diagram \ref{lift}. What we need to show is that the bottom horizontal
map in diagram \ref{lift} lifts as well.

Since $\bCL(X^{d+e-1},X)$ is defined via a pullback diagram \ref{LinternalDef} (with $d+e-1$
instead of $d$), we need to show that there is a commutative  diagram of the form
\begin{Statement} \label{thm4-diagram}
$$\xymatrix{\bHom(X^d,X) \times \bHom(X^e,X) \ar[rr]^{-\circ_i-} & & \bHom(X^{d+e-1},X) \ar[r]
&  \prod_{i=1}^{d+e-1}  \bHom(X^d,X^d) \\
\bCL(X^d,X) \times \bCL(X^e,X) \ar[u] \ar[rrr] & & &  \prod_{i=1}^{d+e-1}  \bAut(X^d) \ar[u] 
}$$
\end{Statement} 
Combining Lemma \ref{lem-3} (with $d+e-1$ instead of $d$) and diagram \ref{lift} yields a
commutative diagram
$$\xymatrix{\Hom_\CC(X^d,X) \times \Hom_\CC(X^e,X) \ar[rr]^{-\circ_i-} & & \Hom_\CC(X^{d+e-1},X) \ar[r]
&  \prod_{i=1}^{d+e-1}  \Hom_\CC(X^d,X^d) \\
\CL_\CC(X^d,X) \times \CL_\CC(X^e,X) \ar[u] \ar[rrr] & & &  \prod_{i=1}^{d+e-1}  \Aut_\CC(X^d) \ar[u] 
}$$
The top horizontal and two vertical maps in this diagram lift to maps as in the diagram \ref{thm4-diagram}.
The hypothesis on lifting maps through morphisms of the form $\bAut(Y)\to$ $\bHom(Y,Y)$
applied to the  $d+e-1$ components on the right side of the diagram \ref{thm4-diagram} shows the existence
of the lower horizontal map making the diagram \ref{thm4-diagram} commutative. The uniqueness of such a map
follows from the fact that the right vertical map is a monomorphism, where we use that the
functor $\Hom_\CC(\One,-)$ is faithful.
\end{proof}

\begin{Remark} \label{Rem-thm4} 
We do not
know whether the hypothesis on lifting morphisms $Z\to \bHom(Y,Y)$ through $\bAut(Y)\to$ $\bHom(Y,Y)$
is indeed needed for Theorem \ref{thm4} to hold. 
This hypothesis holds in the category of compactly generated topological spaces. 
It is easy to see that this hypothesis holds if $\bAut(Y) \to$ $\bHom(Y,Y)$  is a regular
monomorphism (these are  monomorphisms which are an equaliser of a pair of parallel morphisms),
assuming as before  that $\Hom_\CC(\One,-)$ is faithful.
\end{Remark}

\section{Latin hypercubes in terms of pullback diagrams and further remarks} \label{pullbackSection} 

Definition \ref{hypercubeDef} describes  Latin hypercubes 
of dimension $d$ over a non-empty set $X$ as subsets of $X^{d+1}$ instead as the 
graph of a  function $X^d\to X$.  
We  describe the composition maps $-\circ_i -$  in terms  of these subsets as pullbacks.

\begin{Proposition} \label{iProp}
Let $X$ be a non-empty set, and let $d$, $e$, $i$  be positive integers such 
that $1\leq i\leq d$. 
Let $L\subset X^{d+1}$ and  $M\subseteq X^{e+1}$ be Latin hypercubes over $X$
of dimension $d$ and $e$, respectively.
Let $f : X^d\to X$ and $g : X^e \to X$ be the maps whose graphs are $L$ and $M$,
respectively. Denote by $L\circ_i M\subseteq X^{d+e}$  the Latin hypercube over $X$ 
of dimension $d+e-1$ which is the graph of the map $f\circ_i g : X^{d+e-1}\to X$.

An element $(x_1,x_2,..,x_{d+e})\in$ $X^{d+e}$ belongs to $L\circ_i~M$ if and only if 
there is an element $z\in X$ such that $(x_1,..,x_{i-1},z,x_{e+i},..,x_{d+e})\in$ $L$ and 
such that $(x_i,..,x_{e+i-1},z)\in$ $M$.
Then $z$ is uniquely determined by the elements $x_1$, $x_2$,..,$x_{d+e}$.
Equivalently, we have a pullback diagram of the form
$$\xymatrix{ L\circ_i M \ar[rr]^\beta \ar[d]_\alpha  & &  M \ar[d]^{\mu_{e+1}} \\
L \ar[rr]_{\lambda_i} & & X}$$
where $\lambda_i$ is the restriction to $L$ of the canonical projection 
$\pi^{d+1}_i : X^{d+1}\to X$ onto the coordinate $i$, where $\mu_{e+1}$ is the restriction 
to $M$ of the canonical projection $\pi^{e+1}_{e+1} : X^{e+1}\to X$ onto the coordinate 
$e+1$, and where 
$$\alpha(x_1,x_2,..,x_{d+e}) = (x_1,..,x_{i-1},z,x_{e+i},..,x_{d+e}),$$
$$\beta(x_1,x_2,..,x_{d+e}) = (x_i,..,x_{i+e-1},z).$$
\end{Proposition}

\begin{proof}
Note that  $L\circ_i M$ is indeed a Latin hypercube  by Theorem \ref{thm1}.
We have 
$$(x_1,x_2,..,x_{d+e})\in L\circ_i M$$ 
if and only if
$$x_{d+e} = (f\circ_i g)(x_1,x_2,..,x_{d+e-1}) = 
f(x_1,..,x_{i-1}, g(x_i,..,x_{i+e-1}),x_{i+e},..,x_{d+e-1}).$$
Since 
$z=g(x_i,..,x_{e+i-1})$
is the unique element in $X$ such that $(x_i,..,x_{i+e-1},z)\in M$ and the unique
element such that $(x_1,..,x_{i-1},z,x_{i+e},..,x_{d+e})\in L$, the first statement follows.
The second statement follows from the fact that $z$ is uniquely determined by the 
coordinates $x_i$, $1\leq i \leq d+e$ 
\end{proof}

 It is well-known that for a  Latin hypercube of dimension $d\geq 2$ 
over a non-empty set $X$, fixing one of the coordinates in $X^{d+1}$ yields
a Latin hypercube of dimension $d-1$ (this was implicitly used in the proof of
Theorem \ref{thm1}). This can be  extended to Cartesian monoidal categories
with pullbacks. We will need the following Lemma to show this.

\begin{Lemma} \label{monoidalLemma1}
Let $\CC$ be a Cartesian monoidal category. Let $X$ be an object in $\CC$ and
let $d$, $s$  be  positive integers such that $s\leq d$ and $d\geq 2$. 
Let $c :\One\to X$ be
a morphism in $\CC$. Denote by $\Id_{X^d}\times c : X^d\to X^{d+1}$ the 
unique morphism satisfying $\tau^{d+1}_{d+1} \circ (\Id_{X^d}\times c) =$ $\Id_{X^d}$
and $\pi^{d+1}_{d+1} \circ (\Id_{X^d}\times c) =$ $t\times c$, where $t$ is the
unique morphism $X^d\to\One$, and where we identify $X^d=X^d\times\One$ and
$X=\One\times X$.
The diagram
$$\xymatrix{ X^d \ar[rr]^{\tau^d_s} \ar[d]_{\Id_{X^d}\times c} & & 
X^{d-1} \ar[d]^{\Id_{X^{d-1}}\times c} \\
X^{d+1} \ar[rr]_{\tau^{d+1}_s} & & X^d} $$
is a pullback  diagram.
\end{Lemma}

\begin{proof} 
By permuting the coordinates it suffices to show this for $s=d$. 
Let $Z$ be an object in $\CC$, and let $u : Z\to X^{d+1}$ and
$v : Z\to X^{d-1}$ be morphisms such that 
$$\tau^{d+1}_d\circ u= (\Id_{X^{d-1}}\circ c)\circ v.$$
We need to show that there is a unique morphism $w : Z\to X^d$ satisfying
$u=(\Id_{X^d}\times c)\circ w$ and $v = \tau^d_d\circ w$. By checking on
coordinates one sees that $w = (v,u_d) : Z\to X^{d-1}\times X=X^d$ is
the unique morphism with this property, where as before 
$u_d=$ $\pi^{d+1}_d\circ u$.
\end{proof}

\begin{Proposition} \label{monoidalProp2}
Let $\CC$ be a Cartesian monoidal category with  pullbacks.
Let $d$ be an integer such that $d\geq 2$.  Let $\lambda : L\to X^{d+1}$ be
a morphism in $\CC$. 
Suppose that the morphism $\lambda : L \to X^{d+1}$ is a Latin hypercube.
Then the morphism $\tau^{d+1}_{d+1}\circ\lambda : L\to X^d$ is an isomorphism, and
for every morphism $c : \One \to X$ and every pullback diagram of the form
$$\xymatrix{L_c \ar[rr]^{\lambda_c} \ar[d] & & X^d \ar[d]^{(\Id_{X^d})\times c}\\
L \ar[rr]_\lambda & & X^{d+1}}$$
the morphism $\lambda_c : L_c \to X^d$ is a Latin hypercube.
\end{Proposition} 

\begin{proof} 
The morphism $\tau^{d+1}_{d+1}\circ\lambda$ is an isomorphism as part of the
definition of a Latin hypercube.
Let $c : \One \to X$ be a morphism, and let 
 $s$ be an integer such that $1\leq s\leq d$. Consider a pullback diagram as in the statement.
 We need to show that $\tau^d_s\circ\lambda_c : L_c \to X^{d-1}$ is an isomorphism. 
Since $s\leq d$ and $d\geq 2$ we can
complete the diagram in the statement to a commutative diagram
$$\xymatrix{L_c \ar[rr]^{\lambda_c} \ar[d] & & 
X^d \ar[rr]^{\tau^d_s} \ar[d]^{\Id_{X^d}\times c} & & X^{d-1} \ar[d]^{\Id_{X^{d-1}}\times c}\\
L \ar[rr]_\lambda & & X^{d+1} \ar[rr]_{\tau^{d+1}_s} & & X^d}$$
The right square is a pullback diagram by Lemma \ref{monoidalLemma1}. Since the left square
is a pullback diagram, the pasting law for pullbacks implies that the outer rectangle
is a pullback diagram as well. Since the lower horizontal morphism in the outer
rectangle is an isomorphism, so is the upper horizontal morphism (we use here that
pullbacks preserve isomorphisms). 
\end{proof}

We conclude this note with some further remarks.

\begin{Remark}
The notion of transversals can be adapted to the situation of Definition 
\ref{hypercubeDef} as follows.
Given an object $X$ in a category with finite products and a positive integer $d$, 
a {\em transversal in} $X^{d+1}$ is a morphism $\sigma : X\to X^{d+1}$ such that
$\sigma_i=\pi^{d+1}_i\circ\sigma$ is an automorphism of $X$, for $1\leq i\leq d+1$.
The morphism $\sigma$ is obviously a monomorphism. 
We say that such a transversal is {\em contained in a Latin hypercube} 
$\lambda : L\to X^{d+1}$ if there is a morphism $\iota : X\to L$ such that
$\lambda\circ\iota=$ $\sigma$. In that case, $\iota$ is necessarily a monomorphism
as well.  If $X$ is a non-empty set, then a transversal $\sigma$
can be identified with the subset $\sigma(X)$ of $X^{d+1}$. If $X$ is finite and has
an abelian group structure, and if $T\subseteq X^{d+1}$ is a transversal, then
we have a version of the $\Delta$-Lemma \cite[Lemma 2.1]{Wan} as follows.
Denote by $\alpha : X^{d+1}\to X$ the alternating sum map
$\alpha(x_1,x_2,..,x_{d+1})=\sum_{i=1}^{d+1} (-1)^{i-1}x_i$ and by $t$ the
sum of all involutions in $X$. Since the sum of all elements in $X$ is equal
to the sum $t$ of all involutions in $X$, an elementary calculation shows that
$$\sum_{x \in T} \ \ \alpha(x) = 
\begin{cases}
0 &  \text{if}\ d \ \text{is odd} \\[3pt]  
t &  \text{if}\ d\ \text{is even.}
\end{cases} $$
If $X=\Z/n\Z$ for some positive integer $n$, then $t=0$ if $n$ is odd, and $t=$ $\frac{n}{2} + n\Z$
if $n$ is even. 
\end{Remark}

\begin{Remark}
Adapting another well-known notion, the {\em graph of a Latin hypercube}
$\lambda : L\to X^{d+1}$ of dimension $d$ over an object $X$ in a Cartesian
monoidal category $\CC$ is the graph with vertex set $\Hom_\CC(\One,L)$,
with an edge  between two morphisms $\eta, \theta : \One\to L$ if there exists
an index $i$ such that $1\leq i\leq d+1$ and such that
$\pi^{d+1}_i\circ\lambda\circ \eta = \pi^{d+1}_i\circ\lambda\circ\theta.$
In the case where $X$ is a non-empty set and $L$ a subset of $X^{d+1}$
this is equivalent to stating that two elements in $L$
are connected by an edge if and only if they have a common coordinate. 
(Being a Latin
hypercube implies that if two elements of $L$ have two common coordinates, then
these two elements are equal.)
\end{Remark}

\begin{Remark}
Latin hypercubes of a fixed dimension $d$ over objects in a category $\CC$ with finite
products form themselves  a category. Let $X$, $Y$ be objects in $\CC$. Let
$\lambda : X^d\to X^{d+1}$ and $\mu : Y^d \to Y^{d+1}$ be Latin hypercubes.
A morphism from $\mu$ to $\lambda$ is a morphism $\iota : Y\to X$ such that
$\lambda\circ \iota^{\times d} =$ $\iota^{\times(d+1)}\circ\mu$. Here $\iota^{\times d} :
Y^d\to X^d$ is the morphism obtained by taking $d$ times the product of $\iota$;
similarly for $d+1$. Note though that the notion of isomorphism in this category is
different from the notion of isomorphism considered in the Introduction, 
where we regard Latin hypercubes of dimension $d$ over $X$   as objects in the 
 under-category of $X^{d+1}$. For a Latin hypercube
$\lambda : X^d\to X^{d+1}$ of the form $\lambda = (\Id_{X^d}, f)$ for some
morphism $f : X^d\to X$, one checks easily that 
an automorphism of this Latin hypercube in the category
defined here  is an automorphism $\iota\in\Aut_\CC(X)$ satisfying 
 $\iota\circ f=$ $f\circ \iota^{\times d}$.
\end{Remark}

\bigskip\noindent
{\em Acknowledgement.} The present paper has been partially supported by
EPSRC grant EP/X035328/1.



\begin{thebibliography}{WWW}

\bibitem{BCCS} R. A. Bailey, P. J. Cameron, C. E. Praeger, and C. Schneider,
{\em The geometry of diagonal groups.} Trans. Amer. Math. Soc. {\bf 375}
(2022),  5259--5311.


\bibitem{Kelly82} G. M. Kelly, {\em Basic concepts of enriched category theory}.
Cambridge University Press, London Mathematical Society Lecture
Notes Series {`bf 64} (1982).

\bibitem{Kelly05} G. M. Kelly, {\em  On the operads of J. P. May.} 
Repr. Theory Appl. Categ. No. {\bf 13}  (2005), 1--13.


\bibitem{MacMoe} S. Mac Lane and I. Moerdijk, {\em Sheaves in Geometry and Logic}. 
Springer-Verlag New York (1992).

\bibitem{May72} P. May, {\em The geometry of iterated loop spaces.} 
Lectures Notes in Mathematics {bf 271}, Springer-Verlag Berlin  (1972).

\bibitem{MarShnSta} M. Markl, S. Shnider, and J. Stasheff, 
{\em Operads in Algebra, Topology, Physics.} Math. Surveys and Monographs {\bf 96},
Amer. Math. Soc. (2002). 

\bibitem{McKayWan} B. D. McKay and I. M. Wanless, {\em A census of small Latin
hypercubes}. SIAM J. Discrete Math. {\bf 22} (2008), 719--736.

\bibitem{Wan} I. M. Wanless, {\em Transversals in Latin Squares: a survey.} Surveys in
combinatorics 2011, 403--437, London Math. Soc. Lecture Note Ser., {\bf 392}, 
Cambridge University Proess, Cambridge (2011).


\end{thebibliography}
\end{document}